\documentclass[preprint, aoas, dvipsnames]{imsart}

\usepackage{amsthm, amsmath, amsfonts, amssymb, bbm}
\usepackage{algpseudocode}
\usepackage{algorithm}
\usepackage{multicol, multirow}

\theoremstyle{definition}
\newtheorem*{example}{Example}
\theoremstyle{plain}
\newtheorem{theorem}{Theorem}[section]
\newtheorem{corollary}[theorem]{Corollary}
\newtheorem{lemma}[theorem]{Lemma}
\newtheorem{proposition}{Proposition}[section]
\theoremstyle{remark}
\newtheorem{remark}{Remark}

\usepackage{graphicx, tikz}
\usetikzlibrary{decorations.pathmorphing,decorations.pathreplacing,decorations.shapes,mindmap,backgrounds}
\usepackage{trimclip}
\usepackage{caption}
\usepackage{natbib}

\startlocaldefs
\renewcommand{\cfrac}[2]{\ \begin{array}{c}\multicolumn{1}{c|}{#1}\\ \hline\multicolumn{1}{|c}{#2}\end{array}\ }

\newcommand{\specialcell}[2][c]{\begin{tabular}[#1]{@{}c@{}}#2\end{tabular}}
\endlocaldefs

\RequirePackage{hyperref}

\begin{document}

\begin{frontmatter}

\title{Optimal Uncertainty Quantification on moment class using Canonical Moments}
\runtitle{Canonical Moments for OUQ}


\author{\fnms{J\'er\^ome} \snm{Stenger$^{1,2}$}\corref{}\ead[label=e1]{jerome.stenger@edf.fr}},
\author{\fnms{Fabrice} \snm{Gamboa$^{1}$}\ead[label=e4]{fabrice.gamboa@math.univ-toulouse.fr}},
\author{\fnms{Merlin} \snm{Keller$^{2}$}\ead[label=e2]{merlin.keller@edf.fr}},
\and
\author{\fnms{Bertrand} \snm{Iooss$^{1,2}$}\ead[label=e3]{bertrand.iooss@edf.fr}}

\address{EDF R$\&$D \\ 6 quai Watier\\ 78401 Chatou \\ \printead{e1} \\ \printead{e2,e3}}
\address{Universit\'e Paul Sabatier\\ 118 Route de Narbonne \\ 31400 Toulouse \\ \printead{e4}}
\affiliation{$^1$Universit\'e de Toulouse Paul Sabatier, $^2$EDF R$\&$D}
\runauthor{Stenger et al.}

\begin{abstract}
We gain robustness on the quantification of a risk measurement by accounting for all sources of uncertainties tainting the inputs of a computer code. We evaluate the maximum quantile over a class of distributions defined only by constraints on their moments. The methodology is based on the theory of canonical moments that appears to be a well-suited framework for practical optimization. 
\end{abstract}



\end{frontmatter}


\section{Introduction} 

Uncertainty quantification methods address problems related to  with real world variability. Generally, an engineering system is represented by a numerical function $Y=G(X)$ \citep{smith_uncertainty_2014}, whose inputs $X\in\mathbb{R}^p$ are uncertain and modeled by random variables. The variable of interest is the scalar output of the computer code, but the statistician rather work with some quantities of interest, for example, a quantile, a probability of failure, or any measures of risk. Uncertainty quantification aims to characterize how the variability of a system and its model affect the quantity of interest (\cite{derocquigny_model_2012}, \cite{baudin_title_2017}).

We propose to gain robustness on the quantification of this measure of risk. Usually input values are simulated from an associated joint probability distribution. This distribution is often chosen in a parametric family, and its parameters are estimated using a sample and/or the opinion of an expert. However the difference between the probabilistic model and the reality induces uncertainty. The uncertainty on the input distributions is propagated to the quantity of interest, as a consequence, different choices of input distributions will lead to different values of the risk measures.

To consider this uncertainty, we propose to evaluate the maximum risk measure over a class of distributions. Different classes are suggested in the literature mainly discussed in the work of Berger and Hartigan in the context of Robust Bayesian Analysis (see \cite{ruggeri_robust_2005}). They consider for example the generalized moment set \citep{deroberts_bayesian_1981} or the $\varepsilon$-contamination set \citep{sivaganesan_ranges_1989}. The generalized moment set has some really nice properties studied by \cite{winkler_extreme_1988} based on the well known Choquet theory \citep{choquet_lectures_2018}. An extension of Winkler's work has been more recently published by \cite{owhadi_optimal_2013} under the name of Optimal Uncertainty Quantification (OUQ). In this paper we will focus on classes of measures specified by classical moment constraints. This is a particular case of the framework introduced by \cite{owhadi_optimal_2013} justified by our industrial context, mainly related to nuclear safety issues (\cite{wallis_uncertainty_2004}, \cite{prosek_nuclear_2007}). Indeed, in practice the estimation of the input distributions, built with the help of the expert, often relies only on the knowledge of the mean or the variance of the input variables. 

One of the main problems is the computational complexity of the optimization of the risk measure over the given class of distribution. In the moment context, Semi-Definite-Programming \citep{henrion_gloptipoly_2009} has been already already explored by \cite{betro_robust_2005} and \cite{lasserre_moments_2010}, but the deterministic solver rapidly reaches its limitation as the dimension of the problem increases. One can also find in the literature a Python toolbox developed by \cite{mckerns_optimal_2012} called Mystic framework that fully integrates the OUQ framework. However, it was built as a generic tool for generalized moment problems and the enforcement of the moment constraints is not optimal. By restricting the work to classical moment sets, we propose an original and practical approach based on the theory of canonical moments \citep{dette_theory_1997}. Canonical moments of a measure can be seen as the relative position of its moment sequence in the moment space. It is inherent to the measure and therefore present many interesting properties. The performance of our algorithms exceeds all other methods so far.

The paper proceeds as follows. Section \ref{sec: OUQ principles} describes the OUQ framework and the OUQ reduction theorem. A short introduction to the theory of canonical moments is then presented in Section \ref{sec: Theory and methodology}. The algorithms and the methodology used for practical optimization are presented in Section \ref{sec: Algorithms}. Section \ref{sec: Numerical tests on toy example} and \ref{sec: Real case study} are dedicated to the presentation of two numerical tests, one is a toy example, the other is a real case. Section \ref{sec: Conclusion} gives some conclusions and perspectives.

\section{OUQ principles} 
\label{sec: OUQ principles}

\subsection{Reduction theorem} 
In this work, we consider the quantile of the output of a computer code $G:\mathbb{R}^p\rightarrow \mathbb{R}$, seen as a black box function. In order to gain robustness on the risk measurement, our goal is to find the maximal quantile over a class of distributions. More precisely we study
\begin{eqnarray}
    \overline{Q}_\mathcal{A}(\alpha) & = & \sup_{\mu \in \mathcal{A}} \bigg[ \inf \left\lbrace h  \in \mathbb{R}\ | \ F_{\mu}(h) \geq \alpha \right\rbrace \bigg] \label{eq : Objective value} \ , \\
                                & = & \sup_{\mu \in \mathcal{A}} \bigg[ \inf \left\lbrace h  \in \mathbb{R}\ | \ P_{\mu} (G(X) \leq h) \geq \alpha \right\rbrace \bigg] \ , \nonumber
\end{eqnarray}
where $ \mathcal{A}$ is the optimization set, i.e a subset of all probability measures. Our assumption is that all the scalar inputs of the code are bounded and mutually independent. Further, we enforce $N_i$ moment constraints on the distribution $\mu_i$ of the $i$th parameter. Hence, we have
\begin{eqnarray}
	\mathcal{A} & = & \bigg\{ \mu = \otimes \mu_i \in \bigotimes_{i=1}^{p} \mathcal{M}_1 ([l_i,u_i])\; | \; \mathbb{E}_{\mu_i}[x^j] = c_{j}^{(i)} \ ,  \label{eq : Optimization set} \\ & & \quad  c_{j}^{(i)}\in\mathbb{R}, \text{ for } 1\leq j\leq N_i \text{ and } 1\leq i\leq  p \bigg\}\ , \nonumber
\end{eqnarray}
The objective value \eqref{eq : Objective value} is not very convenient to work with, we transform the expression so that the OUQ reduction theorem can be applied (Theorem \ref{th : OUQ reduction theorem} \citep{owhadi_optimal_2013}). The following result, illustrated in Figure \ref{fig : Duality transformation}, can be interpreted as a duality transformation of our problem. The proof is obvious and is postponed to Appendix \ref{app : proof}. 
\begin{proposition} The following duality result holds
    \[ \overline{Q}_\mathcal{A}(p) = \inf \left\lbrace h  \in \mathbb{R}\; | \; \inf_{\mu\in\mathcal{A}} F_{\mu}(h) \geq p \right\rbrace\ . \]
    \label{THM : DUALITY THEOREM}
\end{proposition}
\begin{figure}[htb]
    \centering
     \begin{tikzpicture}[xscale=2.5, yscale=3.5]
        \draw[->] (0,0) -- (3.2,0);
        \draw[->] (0,0) -- (0,1.05);
        \node[below left] at (0,0) {$0$};
        \node[left] at (0,1) {$1$};
        \node[left] at (0,0.4) {$F_\mu (h)$};
        \node[below right] at (3.2,0) {$h$};
        \draw[domain=0:1.92, color=orange, samples=200] plot(\x,{min(0.829*sqrt(sqrt(\x)) - 0.05*sin(deg(3*\x)),1)});
        \draw[color=orange] (1.92,1) -- (3.2,1);
        \draw[domain=0:3.2, color=blue, samples=100] plot(\x,{min(0.7*ln(\x+1)+0.05*sin(deg(3*\x)),1)});
        \draw[domain=0:3.2, color=ForestGreen] plot(\x,{0.04*exp(\x)-0.04});
        \draw[dashed] (0,0.75) -- (2.98,0.75);
        \draw[dashed] (0.8,0) -- (0.8,0.75);		
        \draw[dashed] (1.98,0) -- (1.98,0.75);
        \draw[dashed] (2.98,0) -- (2.98,0.75);
        \node[left] at (0,0.75) {$p$};
        \node[below, color=orange] at (0.8,0) {$x_2$};
        \node[below, color=blue] at (1.98,0) {$x_1$};
        \node[below, color=ForestGreen] at (2.98,0) {$x_{max}$};
        \draw (2.6,0.2) node[below] {$\inf F_{\mu}(h)$} to[out=100, in=-10] (2.385,0.365);
    \end{tikzpicture}
    \caption{Illustration of the duality result \ref{THM : DUALITY THEOREM}. The orange and the green curves represent the CDF envelope; we can see that the maximum quantile $x_{max}$ is actually the quantile of the lowest CDF.}
	\label{fig : Duality transformation}
\end{figure}
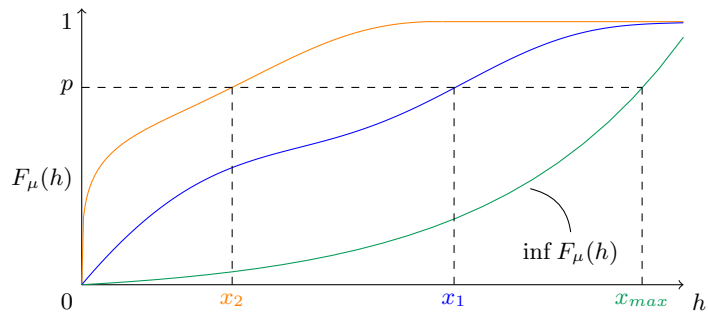

Proposition \ref{THM : DUALITY THEOREM} shows that computing the lowest CDF will easily provide the maximum quantile. Our problem is therefore to evaluate the lowest probability of failure $\inf_{\mu \in \mathcal{A}} P_{\mu}(G(X)\leq h)$ for a fixed threshold $h$. Under this form the OUQ reduction theorem applies (see \cite{owhadi_optimal_2013}, \cite{winkler_extreme_1988}). It states that the optimal solution of our optimization problem is a product of discrete measures. The most general form of the theorem reads as follows:

\begin{theorem}[OUQ reduction {\citet[p.37]{owhadi_optimal_2013}}]
	Suppose that $\mathcal{X} := \mathcal{X}_1 \times \dots \times \mathcal{X}_p$ is a product of Radon spaces. Let 
	\[\mathcal{A} := \left\lbrace (G,\mu)\ \begin{array}{|l}
		G:\mathcal{X}\rightarrow \mathcal{Y} \text{, is a real valued measurable function}, \\
		\mu = \mu_1 \otimes \dots \otimes \mu_p \in \bigotimes_{i=1}^p \mathcal{M}_1(\mathcal{X}_i) \ ,\\
		\text{for each G, and for some measurable functions }\\ \varphi_l:\mathcal{X}\rightarrow \mathbb{R} \text{ and } \varphi_j^{(i)}:\mathcal{X}_i \rightarrow \mathbb{R}\ , \\
		\qquad \qquad \bullet\ \mathbb{E}_{\mu} [\varphi_l] \leq 0 \text{ for } l=1,\dots,N_0\ ,  \\
		\qquad \qquad \bullet\ \mathbb{E}_{\mu_i}[\varphi_j^{(i)}] \leq 0 \text{ for } j=1, \dots, N_i \text{ and } i=1,\dots, p 
	\end{array}\right\rbrace\]
	Let $\Delta_n(\mathcal{X})$ be the set of all discrete measure supported on at most $n+1$ points of $\mathcal{X}$, and
	\[\mathcal{A}_{\Delta} := \left\lbrace (G,\mu) \in \mathcal{A} \ |\ \mu_i \in \Delta_{N_0+N_i}(\mathcal{X}_i) \right\rbrace \ .\]
	Let $q$ be a measurable real function on $\mathcal{X}\times \mathcal{Y}$. Then
	\[\displaystyle \sup_{(G,\mu)\in\mathcal{A}} \mathbb{E}_{\mu}[q(X,G(X))] = \sup_{(G,\mu)\in\mathcal{A}_{\Delta}} \mathbb{E}_{\mu}[q(X,G(X))]\ .\]
	\label{th : OUQ reduction theorem}
\end{theorem}

The heuristic of Theorem \ref{th : OUQ reduction theorem} is that if you have $N_k$ pieces of information relevant to the random variable $X_k$ then it is enough to pretend that $X_k$ takes at most $N_k + 1$ values in $\mathcal{X}_k$. This important point comes from \cite{winkler_extreme_1988}, who has shown that the extreme measures of moment class $\left\lbrace \mu \in \mathcal{M}_1(\mathcal{X}) \ | \ \mathbb{E}_\mu [\varphi_1] \leq 0 , \dots , \mathbb{E}_\mu [\varphi_n] \leq 0 \right\rbrace$ are the discrete measures that are supported on at most $n+1$ points. The work of Winkler is based on the Choquet theory \citep{choquet_lectures_2018} which requires the set to be convex, as moment classes. The strength of Theorem \ref{th : OUQ reduction theorem} is that it extends the result to a tensorial product of moment sets. The proof relies on a recursive argument using Winkler's classification on every set $\mathcal{X}_i$. We recall that the product of spaces traduces the independence of the input variables. Another remarkable fact is that the theorem remains true whatever the function $G$ and the quantity of interest $q$ are. In fact it is only required that the quantity to be optimized is an affine function of the underlying measure $\mu$.

To be more specific, Theorem \ref{th : OUQ reduction theorem} states that the solution to our optimization problem is located on the $p$-fold product of finite convex combinations of Dirac masses.
\[\mathcal{A}_\Delta  = \left\lbrace \mu\in\mathcal{A} \; | \; \mu_i = \sum_{k=1}^{N_i+1} w_{k}^{(i)} \delta_{x_{k}^{(i)}}\ \text{ for }\  1\leq i\leq p\right\rbrace\ ,\]
such that $\overline{Q}_\mathcal{A}(p)=\overline{Q}_{\mathcal{A}_\Delta}(p)$.

Theorem \ref{th : OUQ reduction theorem} drastically simplifies the computational implementation to solve the optimization problem. Indeed, our optimization problem could be rewritten as
\begin{eqnarray}
\inf_{\mu \in \mathcal{A}_\Delta} F_{\mu}(h)\ & = & \inf_{\mu \in \mathcal{A}_\Delta} P_{\mu} (G(X)\leq h) \ , \nonumber \\ & = & \inf_{\mu \in \mathcal{A}_\Delta} \sum_{i_1=1}^{N_1+1} \dots \sum_{i_p=1}^{N_p+1} \omega_{i_1}^{(1)} \dots \omega_{i_p}^{(p)}\ \mathbbm{1}_{\{G(x_{i_1}^{(1)}, \dots, x_{i_p}^{(p)}) \leq h\}} \ ,
\label{eq : Probability of faillure sum of weights}
\end{eqnarray}
thus the weights and positions of the input distributions provide a natural parameterization of the optimization problem. 

\subsection{Limitations}
The reason behind the restriction to equality constraint on the inputs \eqref{eq : Optimization set}, is that it greatly simplify the parameterization of the optimization problem. Suppose that we enforce $N_i$ constraints on the moments of a scalar measure $\mu_i = \sum \omega_j^{(i)} \delta_{x_j^{(i)}}$, Theorem \ref{th : OUQ reduction theorem} states that the solution to our problem is supported by at most $N_i+1$ points. A noticeable fact is that as soon as the $N_i+1$ support points of the distribution are set, the corresponding weights are uniquely determined. Indeed, the $N_i$ constraints lead to $N_i$ equations, and one last equation derives from the measure mass equals to 1. The following $N_i+1$ linear equations holds
\begin{equation}
\left\lbrace \begin{array}{lclcll}
\omega_1^{(i)} & + & \dots & + & \omega_{N_i+1}^{(i)} & = 1 \\
\omega_1 x_1^{(i)} & + & \dots & + & \omega_{N_i+1}^{(i)} x_{N_i+1}^{(i)} & = c_1^{(i)} \\
\multicolumn{1}{c}{\vdots}  & & & & \multicolumn{1}{c}{\vdots} & \multicolumn{1}{r}{\vdots \; \;}\\
\omega_1^{(i)} {x_1^{(i)}}^{N_i} & + & \dots & + & \omega_{N_i+1}^{(i)} {x_{N_i+1}^{(i)}}^{N_i} & = c_{N_i}^{(i)} 
\end{array} \right.
\label{eq : Vandermonde weight system}
\end{equation} 
The determinant of the previous system is a Vandermonde matrix. Hence, the system is invertible as long as the $(x_j^{(i)})_j$ are distinct. The optimization problem can therefore be parameterized only by the support points of the inputs. 

However, in order to proceed to the optimization problem presented in Equation \eqref{eq : Probability of faillure sum of weights}, one must be able to generate positions points that correspond to the support of an admissible measure. That is, a measure that respects the constraints enforced in Equation \eqref{eq : Vandermonde weight system}, with the additional condition $0 \leq \omega_j^{(i)} \leq 1$. 

We define the moment space $M:=M(a,b)=\left\lbrace \mathbf{c}(\mu) \ | \ \mu\in\mathcal{M}_1([a,b])\right\rbrace$ where $\mathbf{c}(\mu)$ denote the sequence of moments of some measure $\mu$. The $n$th moment space $M_n$ is defined by projecting $M$ onto its first $n$ coordinates, $M_n = \{ \mathbf{c}_n(\mu) = (c_1,\dots, c_n)\; | \; \mu\in\mathcal{M}_1([a,b])\}$. $M_2$ is depicted Figure \ref{fig : Admissible set canonical moments}. In order to perform a numerical optimization using evolutionary or simulated annealing solver, the problem is reformulated as follow: given a moment sequence $\mathbf{c}_{n} = (c_1,\dots, c_{n})\in M_{n}$, one must be able to explore the whole set of admissible measures that has been parameterized with the points of their finite supports. Therefore, one shall generate at most $n+1$ support points of some discrete measure, having $\mathbf{c}_n$ as moment sequences. Canonical moments provide a surprisingly well tailored solution of this problem.

The work on canonical moments was first introduced by \cite{skibinsky_range_1967}. His main contribution covered the original study of the geometric aspect of general moment space \citep{skibinsky_maximum_1977}, \citep{skibinsky_principal_1986}. In a number of further papers, Skibinsky proves numerous other interesting properties of the canonical moments.  \cite{dette_theory_1997} have shown the intrinsic relation between a measure $\mu$ and its canonical moments. They highlight the interest of canonical moments in many areas of statistics, probability and analysis such as problem of design of experiments, or the Hausdorff moment problem \citep{hausdorff_momentprobleme_1923}. In the next section, we briefly introduce the theory of canonical moments, widely using results of \cite{dette_theory_1997}, so the interested readers are referred to this book.

\section{Theory and methodology}
\label{sec: Theory and methodology}
\subsection{Canonical moments on an interval}
We first define the extreme values, 
\begin{align*}
    & c_{n+1}^{+} = \max \left\lbrace c\in \mathbb{R} : (c_1,\dots, c_n,c) \in M_{n+1} \right\rbrace \ ,\\
    & c_{n+1}^{-} = \min \left\lbrace c\in \mathbb{R} : (c_1,\dots, c_n,c) \in M_{n+1} \right\rbrace \ ,
\end{align*}
which represent the maximum and minimum values of the $(n+1)$th moment that a measure can have, when its moments up to order $n$ equal to $\mathbf{c}_n$ . The $n$th canonical moment is then defined recursively as 
\begin{equation}
    p_n=p_n(\mathbf{c})=\frac{c_n-c_n^{-}}{c_n^{+}-c_n^{-}}\ .
\end{equation}
Note that the canonical moments are defined up to the degree $N = N(\mathbf{c}) = \min \left\lbrace \right. n \in \mathbb{N} \ |\ \mathbf{c}_n  \in \partial  M_n \left. \right\rbrace $, and $p_N$ is either $0$ or $1$. Indeed, we know from \citep[Theorem 1.2.5]{dette_theory_1997} that $\mathbf{c}_n \in \partial M_n$ implies that the underlying $\mu$ is uniquely determined, so that, $c_n^+ = c_n^-$. We also introduce the quantity $\zeta_n = (1-p_{n-1})p_n$ that will be of some importance in the following. The very nice properties of canonical moments is that they belong to $[0,1]$ and are invariant by any affine transformation of the support of the underlying measures. Hence, we may restrict ourselves to the case $a=0$, $b=1$.
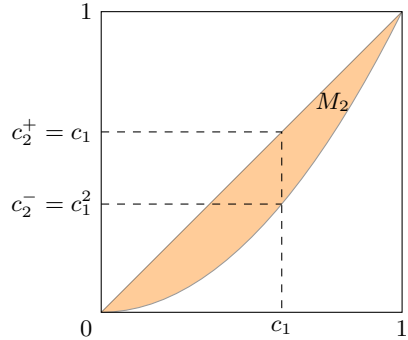
\begin{figure}
    \centering        
    \begin{tikzpicture}[scale=4]
        \draw (0,0) rectangle (1,1);
        \draw[domain=0:1, fill=orange, opacity=0.4] plot(\x,{\x*\x});
        \draw[opacity=0.4] (0,0) -- (1,1);
        \node[left] at (0,1) {$1$};
        \node[below left] at (0,0) {$0$};
        \node[below] at (1,0) {$1$};
        \node at (0.77,0.7) {$M_2$};
        \node[below] at (0.6,0) {$c_1$};
        \node[left] at (0,0.6) {$c_2^+ = c_1$};
        \node[left] at (0,0.36) {$c_2^- = c_1^2$};
        \draw[dashed] (0,0.6) -- (0.6,0.6) -- (0.6,0);
        \draw[dashed] (0,0.36) -- (0.6,0.36);
    \end{tikzpicture}
    \caption{The moment set $M_2$ and definition of $c_2^+$ and $c_2^-$ for $(a,b)=(0,1)$.}
	\label{fig:moment set M2}
\end{figure}

\subsection{Stieltjes Transform and Canonical Moments}
We introduce the Stieltjes Transform, which connects canonical moments to the support of a discrete measure. Then Stieltjes transform of $\mu$ is definides as 
\[S(z)=S(z,\mu) = \int_a^b \frac{d\mu (x)}{z-x}\ , \quad (z\in\mathbb{C}\backslash\{\text{supp}(\mu)\}) \ .\]
The transform $S(z, \mu)$ is an analytic function of $z$ in $\mathbb{C} \backslash \text{supp}(\mu)$. If $\mu$ has a finite support then 
\[S(z)= \int_a^b \frac{d\mu (x)}{z-x}= \sum_{i=1}^n \frac{\omega_i}{z-x_i}\ , \]
where the support points of the measure $\mu$ are distinct and denoted by $x_1,\allowbreak \dots, x_n$, with corresponding weights $\omega_1,\dots, \omega_n$. Alternatively, the weights are given by $\omega_i = \lim_{z \rightarrow x_i} (z-x_i)S(z)$. We can rewrite the transform as a ratio of two polynomials with no common zeros. The zeros of the denominator determine the support of $\mu$. 
\begin{equation}
S(z) = \frac{Q_{n-1}(z)}{P_n^*(z)} \ ,
\label{eq : Stieljes transform Q/P}
\end{equation} 
where $P_n^*(z)= \prod_{i=1}^n (z - x_i)$ and 
\[ \omega_i = \frac{Q_{n-1}(x_i)}{\frac{d}{dx} P_n^*(x) |_{x=x_i}}\ . \]
First, we introduce some basic property of continued fractions
\begin{lemma}
A finite continued fraction is an expression of the form 
\[ b_0+\frac{a_1}{b_1+\frac{a_2}{b_2+ \dots}} = b_0 + \cfrac{a_1}{b_1} + \cfrac{a_2}{b_2} + \dots + \cfrac{a_n}{b_n} = \frac{A_n}{B_n} \ .\]
The quantities $A_n$ and $B_n$ are called the $n$th partial numerator and denominator. There are basic recursive relations for the quantities $A_n$ and $B_n$ given by
\begin{eqnarray*}
A_n & = & b_n A_{n-1} + a_n A_{n-2} \ , \\
B_n & = & b_n B_{n-1} + a_n B_{n-2} \ ,
\end{eqnarray*}
for $n \geq 1$ with initial conditions 
\begin{eqnarray*}
A_{-1}=1 & \quad , \quad & A_0=b_0 \ , \\
B_{-1}=0 & \quad , \quad & B_0=1 \ .
\end{eqnarray*}
\label{Lem : Continued Fraction}
\end{lemma}

The case where the canonical moments are given is important, indeed we have the following result 

\begin{theorem}[{\citet[Theorem 3.3.1]{dette_theory_1997}}]
Let $\mu$ be a probability measure on the interval $[a,b]$ and $z\in \mathbb{C} \backslash [a,b]$, then the Stieltjes transform of $\mu$ has the continued fraction expansion 
\begin{align*}
S(z) & = \cfrac{1}{z-a} \ -\  \cfrac{\zeta_1 (b-a)}{1} \ -\  \cfrac{\zeta_2 (b-a)}{z-a} \ -\  \dots \label{eq : Stieljes transform}\\
& = \cfrac{1}{z-a- \zeta_1 (b-a)} \ -\  \cfrac{\zeta_1 \zeta_2 (b-a)^2}{z-a - (\zeta_2+ \zeta_3)(b-a)} \\ &  -\  \cfrac{\zeta_3 \zeta_4 (b-a)^2}{z-a- (\zeta_4 +\zeta_5) (b-a)} \ -\  \dots \nonumber
\end{align*} 
\label{Th : Stieljes transform}
Where we recall that $\zeta_n := p_{n-1}(1-p_n)$.
\end{theorem}

Theorem \ref{Th : Stieljes transform} states that the Stieltjes transform can be computed when one knows the canonicals moments. It follows from Equation (\ref{eq : Stieljes transform Q/P}) and Lemma \ref{Lem : Continued Fraction} that we have the following recursive formula for $P_n^*$
\begin{equation}
 P_{k+1}^*(x)=(x-a-(b-a)(\zeta_{2k}+\zeta_{2k+1}))P_{k}^*(x)-(b-a)^2 \zeta_{2k-1} \zeta_{2k} P_{k-1}^* \ ,
\label{eq : Recursive formula polynomial}
\end{equation}
where $P_{-1}^*=0$, $P_0^*=1$. The support of $\mu$ is thus the roots of $P_n^*$. This obviously leads to the following theorem.

\begin{theorem}[{\citet[Theorem 3.6.1]{dette_theory_1997}}]
Let $\mu$ denote a measure on the interval $[a, b]$ supported on $n$ points with canonical moments $p_1 , p_2 , \dots$. Then, the support of $\mu$ is the set of $\{x:P_n^*(x)=0\}$
\label{th : Dette Generation of measure}
\end{theorem}

In the following we consider a fixed sequence of moments $\mathbf{c}_n = (c_1, \dots, c_n) \in M_n$, let $\mu$ be a measure supported on at most $n+1$ points, such that its moments up to order $n$ coincide with $\mathbf{c}_n$. $\mathbf{c}_n$ is uniquely related to $\mathbf{p}_n=(p_1, \dots, p_n)$ the corresponding sequence of canonical moments, so that it is equivalent to constraint classical moments or canonical moments. Corollary \ref{Th : Generation of measure} is the moment version of Theorem \ref{th : Dette Generation of measure}. The only difficulty compared to Theorem \ref{th : Dette Generation of measure} is that one try to generate admissible measures supported on at most $n+1$ Dirac masses. Given a measure supported on less than $n+1$ points, the question is therefore to know whether it makes sense to evaluate the $n+1$ roots of $P_{n+1}^*$. A limit argument is used for
 the proof.

\begin{corollary}
    Consider a sequence of moment $\mathbf{c}_n = (c_1, \dots, c_n) \in M_n$, and the set of measure 
    \[ \mathcal{A}_\Delta = \left\lbrace \mu = \sum_{i=1}^{n+1} \omega_i \delta_{x_i} \in \mathcal{M}_1([a,b]) \ | \ \mathbbm{E}_{\mu}(x^j) = c_j, \ j=1,\dots,n \right\rbrace \ .\]
    We define 
    \[ \Gamma=\left\lbrace (p_{n+1}, \dots, p_{2n+1})\in [0,1]^{n+1} \ | \ p_i\in \{0,1\} \Rightarrow p_k=0 ,\ k > i \right\rbrace\ . \]
	Then there exists a bijection between $\mathcal{A}_{\Delta}$ and $\Gamma$.
\label{Th : Generation of measure}
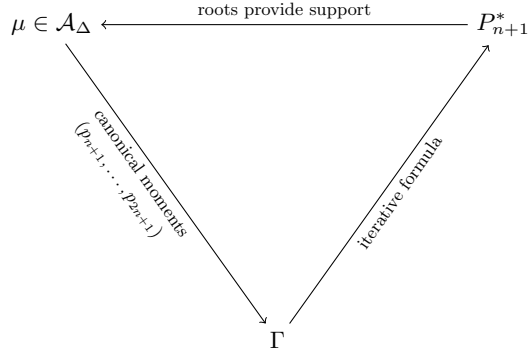
\begin{figure}[htb]
	\centering
    \begin{tikzpicture}[remember picture, scale=3]
		\node (A) at (-1,0) {$\mu\in\mathcal{A}_\Delta$};
		\node (B) at (1,0) {$P_{n+1}^*$};
		\node (C) at (0,-1.4) {$\Gamma$};
		\draw[->] (A) to node[below, rotate=-55, text width=3.5cm, scale=0.75]{canonical moments $(p_{n+1},\dots, p_{2n+1})$} (C) ;
		\draw[->] (C) to node[below, rotate=55, scale=0.75]{{iterative formula}} (B);
		\draw[->] (B) to node[above, scale=0.75]{{roots provide support}} (A);
	\end{tikzpicture}
	\caption{Relation between the set of admissible measures and the canonical moments.}
\end{figure}
\end{corollary}

\begin{proof}
Without loss of generality we can always assume $a=0$ and $b=1$ as the problem is invariant using affine transformation. We first consider the case where $\text{card}(\text{supp}(\mu))$ is exactly $n+1$. From Theorem \ref{th : Dette Generation of measure}, the polynomial $P_{n+1}^*$ is well defined with $n+1$ distinct roots corresponding to the support of $\mu$. Notices that this implies that $(p_1, \dots, p_{2n-1})$ belongs to $]0,1[^{2n-1}$ and that $p_{2n}, p_{2n+1}$ or $p_{2n+2}$ belongs to $\{ 0,1 \}$. 

Now, the functions $g(x,z) = 1/(z-x)$ are equicontinuous for $z$ in any compact region which has a positive distance from $[0,1]$. The Stieljes transform is a finite sum of equicontinuous functions and therefore also equicontinuous. Thus if a measure $\mu$ converges weakly to $\mu^*$, the convergence must be uniform in any compact set with positive distance from $[0,1]$ (see \cite{royden_real_1968}). It is then always possible to restrict ourselves to measures of cardinal $m < n+1$, by letting $p_k$ converge to $0$ or $1$ for $2m-2\leq k \leq 2m $. Note that by doing so the polynomials $P_m^*$ and $P_{n+1}^*$ will have the same roots. But, $P_{n+1}^*$ and $Q_{n-1}$ will have some others roots of multiplicity strictly equals (see Equation (\ref{eq : Stieljes transform Q/P}) and (\ref{eq : Recursive formula polynomial})). The corresponding weights of this roots are vanishing, so that the measures extracted from $P_m^*$ and $P_{n+1}^*$ are the same. 
\end{proof}
\begin{remark}
The first $n$ canonical moments $(p_1, \dots, p_n)$ of $\mu$ are fixed; the free parameters that allow to explore the whole set $\mathcal{A}_{\Delta}$ are the $n+1$ last canonical moments $(p_{n+1},\dots, p_{2n+1})$.
\end{remark}
\begin{remark}
From a computational point of view, as the proof relies on a limit argument, we can always generate $p_k \in ]0,1[\ ,\  \text{ for }  n+1\leq k \leq 2n+1 $. This prevents the condition $p_k \in \{ 0,1\} \Rightarrow p_j = 0$ for $j>k$. 
\end{remark}

\begin{example}
Let $\mu \in \mathcal{P}([0,1])$ be such that $c_1(\mu) = 0.5$ and $c_2(\mu) = 0.35$. The corresponding canonical moments are $p_1 = 0.5$ and $p_2 = 0.4$. We are looking for a measure supported on at most 3 points. Set arbitrarily, $p_3 = 0.2,\ p_4 = 10^{-5}$ and $p_5 = 0.4$\ . The quantity $p_4$ is close to zero, so that, the measure should be theoretically supported on 2 points. The roots of $P_2^*$ are $0.08121$ and $0.73878$ with weights $0.36312$ and $0.63688$, while the roots of $P_3^*$ are $0.08121, 0.73878$ and $0.4$ with corresponding weights $0.36132, 0.63686$ and $2\times10^{-5}$. The last point $0.4$ depends on the choice of $p_5$ but is almost non weighted, the measure converges to the two points distribution as $p_4 \rightarrow 0$.
\end{example}

\begin{figure}[htb]
	\centering
	\clipbox{0pt 0pt 0pt 20pt}{\includegraphics[scale = 0.40]{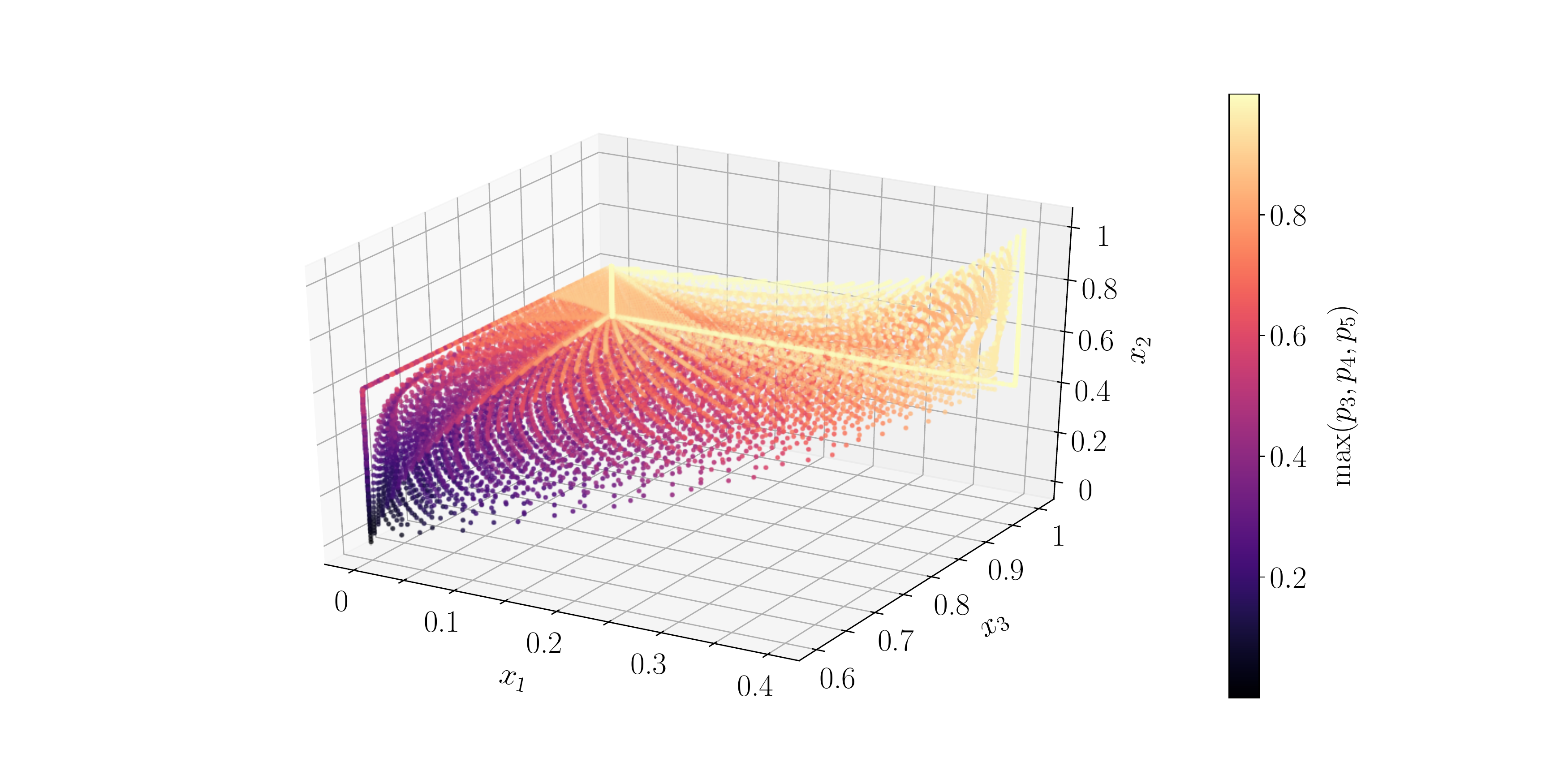}}
	\caption{Visualization of the support generated using a canonical parameterization. Each point coordinates correspond to the three support points (the roots of $P_3^*$) of a measure $\mu$ on $[0,1]$, whose two first moments are set to $c_1=0.5$ and $c_2=0.3$. We generate a regular grid where $p_3, p_4$ and $p_5$ goes from 0 to 1.}
	\label{fig : Admissible set canonical moments}
\end{figure}

\section{Algorithms}
\label{sec: Algorithms}
\subsection{Algorithm for equality constraints}
We now discuss the algorithm used in order to solve the optimization problem \eqref{eq : Probability of faillure sum of weights}, related to the optimization space of Equation \eqref{eq : Optimization set}.
For every $i=1,\dots, p$, the support of measure $\mu_i$ is transformed into $[0,1]$ using the affine transformation $y = l_i +(u_i - l_i)x$. The sequences of moments of the corresponding measures are written $\mathbf{c}'_i = ({c'}_{1}^{(i)}, \dots, {c'}_{N_i}^{(i)})$ where ${c'}_{j}^{(i)}$ reads
\begin{equation}
{c'}_{j}^{(i)} = \frac{1}{(u_i- l_i)^{j}} \sum_{k=0}^{j} \binom{j}{k} (- l_i)^{j -k} c_{k}^{(i)} \ , \quad \text{ for } j = 1,\dots, N_i\, .
\label{eq : Moment's affine transformation}
\end{equation} 
Given a sequence of moments, it is then possible to calculate the corresponding sequence of canonical moments $\mathbf{p}_i = (p_{1}^{(i)} , \dots, p_{N_i}^{(i)})$. \citet[p. 29]{dette_theory_1997} propose a recursive algorithm named \textit{Q-D algorithm} that allows this computation. It drastically fastens the computational time compared to the raw formula that consists of computing Hankel determinants \citep[p. 32]{dette_theory_1997}.
Generally, the OUQ framework involves low order of moments, typically order 2. In this case we dispose of the simple analytical formulas 
\begin{eqnarray*}
    p_1 = c_1 & \quad ,\quad & p_2=\frac{c_2-c_1^2}{c_1(1-c_1)} \ .
\end{eqnarray*} 
As every measure $\mu_i$ is enforced with $N_i$ constraints we are looking for discrete measures written as convex combination of at most $N_i+1$ Dirac masses. Using Theorem \ref{Th : Generation of measure}, the supports of the measures are the roots of $P_{N_i+1}^{*(i)}$ which depends on the $N_i+1$ free parameters $(p_{N_i+1}^{(i)} , \dots, p_{2N_i+1}^{(i)})$. Those correspond to the parameters of the solver. The corresponding weights are then obtain solving equation (\ref{eq : Vandermonde weight system}). Finally, we evaluate the probability of failure (\textsc{p.o.f})  
\[ P_{\mu}(G(X)\leq h) = \sum_{i_1=1}^{N_1+1} \dots \sum_{i_p=1}^{N_p+1} \omega_{i_1}^{(1)} \dots \omega_{i_p}^{(p)}\ \mathbbm{1}_{\{G(x_{i_1}^{(1)} , \dots, x_{i_p}^{(p)}) \leq h\}} \ .\]
The pseudo code presented in algorithm \ref{Alg : p.o.f} summarizes this procedure.

\algrenewcommand\algorithmicrequire{\textbf{Inputs:}}
\begin{algorithm}[ht]
\caption{Calculation of the \textsc{p.o.f}}\label{Alg : p.o.f}
\begin{algorithmic}[1]
\Require  
	\Statex - lower bounds, $\mathbf{l} = (l_1, \dots, l_p)$
	\Statex - upper bounds, $\mathbf{u} = (u_1, \dots, u_p)$
	\Statex - constraints sequences of moments, $\mathbf{c}_i = (c_{1}^{(i)}, \dots, c_{N_i}^{(i)})$ and its corresponding sequences of canonical moments, $\mathbf{p}_i = (p_{1}^{(i)}, \dots, p_{N_i}^{(i)})$ for $1\leq i \leq p$.
	\Statex
	
\Ensure $p_{j}^{(i)}\in ]0,1[$ for $1\leq j \leq N_i$ and $1\leq i \leq p$ 
	\Statex
\Function{\textsc{p.o.f}}{$p_{N_1+1}^{(1)}, \dots, p_{2N_1+1}^{(1)} ,\dots, p_{N_p+1}^{(p)} , \dots, p_{2N_p+1}^{(p)}$}
	\For {$i = 1, \dots, p$}
		\For {$k = 1, \dots N_i$}
			\State $P_{k+1}^{*(i)}=(X-l_i-(u_i-l_i)(\zeta_{2k}^{(i)} +\zeta_{2k+1}^{(i)}))P_{k}^{*(i)}-(u_i-l_i)^2 \zeta_{2k-1}^{(i)} \zeta_{2k}^{(i)} P_{k-1}^{*(i)}$
		\EndFor 
		\State $x_{1}^{(i)} , \dots, x_{N_i+1}^{(i)} = \text{roots}(P_{N_i+1}^{*(i)})$
		\State $\omega_{1}^{(i)}, \dots, \omega_{N_i+1}^{(i)} = \text{weight}(x_{1}^{(i)}, \dots, x_{N_i+1}^{(i)} , \mathbf{c}_i)$
	\EndFor
	\State \Return $\sum_{i_1=1}^{N_1+1} \dots \sum_{i_p=1}^{N_p+1} \omega_{i_1}^{(1)} \dots \omega_{i_p}^{(p)} \ \mathbbm{1}_{\{G(x_{i_1}^{(1)} , \dots, x_{i_p}^{(p)}) \leq h\}}$
\EndFunction
\end{algorithmic}
\end{algorithm}
Notice that the function \textsc{p.o.f} takes $\sum_{i=1}^p (N_i+1)$ arguments. This parameterization presents two main advantages. The first one is that the parameterization is very easy as every input belongs to $]0,1[$. The other one is that every discrete measure satisfies the constraints and is therefore admissible. The only drawback can be the computational complexity to obtain the roots of high degree polynomial. However, this is neglectible compared to the costly part of the algorithm, that is the large number of evaluation of the underlying code $G$, run $\prod_{i=1}^p (N_i+1)$ times.

The \textsc{p.o.f} function can be easily optimized using any global optimization solvers, such as Differential Evolution \citep{price_differential_2005} or Simulated Annealing algorithm \citep{metropolis_equation_1953}, \citep{kirkpatrick_optimization_1983}. 

\subsection{Modified algorithm for inequality constraints}
In the following, we consider inequality constraints for the moments. The optimization set reads
\[\mathcal{A}=\left\lbrace \mu = \otimes \mu_i \in \bigotimes_{i=1}^{p} \mathcal{M}_i([l_i,u_i])\; | \; \alpha_{j}^{(i)} \leq \mathbb{E}_{\mu_i}[x^{j}] \leq \beta_{j}^{(i)}\ , \ 1\leq j\leq N_i \right\rbrace\ . \]
One can notice that $\alpha_{j}^{(i)} \leq \mathbb{E}_{\mu_i}[x^{j}] \leq \beta_{j}^{(i)}$ is equivalent to enforcing two constraints, thus drastically increasing the dimension of the problem. However, it is possible to restrict ourselves to one constraint. onsider the convex function $\varphi_{j}^{(i)} : x \mapsto (x^{j}-\alpha_{j}^{(i)})(x^{j}-\beta_{j}^{(i)})$, Jensen's inequality states that $\varphi_{j}^{(i)}(\mathbb{E}_{\mu_i}(x)) \leq \mathbb{E}_{\mu_i}(\varphi_{j}^{(i)}(x))$. Therefore, the sole constraint $\mathbb{E}(\varphi_{j}^{(i)}(x)) \leq 0$ ensures $\alpha_{j}^{(i)} \leq \mathbb{E}_{\mu_i}[x^{j}] \leq \beta_{j}^{(i)}$. Without loss of generality we still consider measures $\mu_i$ that are convex combination of $N_i+1$ Dirac masses, for $i=1, \dots, p$.

We now propose a modified version of algorithm \ref{Alg : p.o.f} to solve the problem with inequality constraints. For $i = 1, \dots, p$, we denote the moments lower bounds $\boldsymbol{\alpha}_i = (\alpha_{1}^{(i)}, \dots, \alpha_{N_i}^{(i)})$ and the moments upper bounds $\boldsymbol{\beta}_i = (\beta_{1}^{(i)} , \dots, \beta_{N_i}^{(i)})$. We use Equation \eqref{eq : Moment's affine transformation} to calculate the corresponding moment sequence $\boldsymbol{\alpha}'_i$ and $\boldsymbol{\beta}'_i$ after affine transformation to $[0,1]$.
\begin{algorithm}[ht]
\caption{Calculation of the \textsc{p.o.f} with inequality constraints}\label{Alg : p.o.f inequality}
\begin{algorithmic}[1]
\Require  
	\Statex - lower bounds, $\mathbf{l} = (l_1, \dots, l_p)$
	\Statex - upper bounds, $\mathbf{u} = (u_1, \dots, u_p)$
	\Statex - moments lower bounds, $\boldsymbol{\alpha}'_i = ({\alpha'_{1}}^{(i)} , \dots, {\alpha'_{N_i}}^{(i)})$ for $i=1,\dots,p$
	\Statex - moments upper bounds, $\boldsymbol{\beta}'_i = ({\beta'_{1}}^{(i)} , \dots, {\beta'_{N_i}}^{(i)})$ for $i=1,\dots,p$
	\Statex
\Ensure 
	\Statex $p_{j}^{(i)} \in [0,1]$ and ${c'_{j}}^{(i)} \in [{\alpha'_{j}}^{(i)}, {\beta'_{j}}^{(i)}]$ for $1\leq j \leq N_i$ and $1\leq i \leq p$.  
	\Statex
\Function{P.O.F}{${c'_{1}}^{(1)} , \dots, {c'_{N_i}}^{(1)}, p_{N_1+1}^{(1)} , \dots, p_{2N_1+1}^{(1)} ,\dots, {c'_{1}}^{(p)},\dots {c'_{N_p}}^{(p)}, p_{N_p+1}^{(p)}, \dots, p_{2N_p+1}^{(p)}$}
	\For {$i = 1, \dots, p$}
		\For {$k = 1, \dots N_i$}
			\State $p_{k}^{(i)} = f({c'_{1}}^{(i)} , \dots {c'_{k}}^{(i)}) $ \Comment{f transform moments to canonical moments}
		\EndFor
		\For {$k = 1, \dots N_i$}
			\State $P_{k+1}^{*(i)} = (X-l_i-(u_i-l_i)(\zeta_{2k}^{(i)} +\zeta_{2k+1}^{(i)}))P_{k}^{*(i)}-(u_i-l_i)^2 \zeta_{2k-1}^{(i)} \zeta_{2k}^{(i)} P_{k-1}^{*(i)}$
		\EndFor 
		\State $x_{1}^{(i)}, \dots, x_{N_i+1}^{(i)} = \text{roots}(P_{N_i+1}^{*(i)})$
		\State $\omega_{1}^{(i)} , \dots, \omega_{N_i+1}^{(i)} = \text{weight}(x_{1}^{(i)} , \dots, x_{N_i+1}^{(i)} , \mathbf{c}'_i)$
	\EndFor
	\State \Return $\sum_{i_1=1}^{N_1+1} \dots \sum_{i_p=1}^{N_p+1} \omega_{i_1}^{(1)} \dots \omega_{i_p}^{(p)} \ \mathbbm{1}_{\{G(x_{i_1}^{(1)}, \dots, x_{i_p}^{(p)}) \leq h\}}$
\EndFunction
\end{algorithmic}
\end{algorithm}

The \textsc{p.o.f} of algorithm \ref{Alg : p.o.f inequality} has $p + 2\times \sum_{i=1}^p N_i$ arguments. The new parameters are actually the first ${(N_i)}_{|i=1,\dots,p}$th moments of the inputs that were previously fixed. A new step in the algorithm is needed to calculate the canonical moments up to degree $N_i$ for $i=1, \dots, p$. This ensures that the constraints are satisfied while the canonical moments from degree $N_i+1$ up to degree $2N_i+1$ can vary between $]0,1[$ in order to generate all possible measures. The increase of the dimension does not affect the computational times neither the complexity. Indeed, the main cost still arise from to the large number of evaluation of the code $G$, that remains equal to $\prod_{i=1}^p (N_i+1)$. Once again this new \textsc{p.o.f} function can be optimized using any global solver. 

\section{Numerical tests on a toy example} 
\label{sec: Numerical tests on toy example}
\subsection{Presentation of the hydraulic model}
\label{subsec:hydraulic model}
In the following, we address a simplified hydraulic model \citep{keller_estimation_2012}. This code calculates the water height $H$ of a river subject to a flood event. It takes four inputs whose initial joint distribution is detailed in Table \ref{tab:initial distribution hydraulic model}. It is possible to calculate the quantile for those particular distributions. The result is given in Figure \ref{fig : Moment Order Constraints}. However, as we desire to evaluate the robust quantile over a class of measures, we present in Table \ref{tab : Constraints for hydraulic model} the corresponding moment constraints that the variables must satisfy. The constraints are calculated based on the initial distributions, while the bounds are chosen in order to match the initial distributions most representative values.

\begin{table}[ht]
	\centering
	\begin{tabular}{|l|c|}
		\hline
		Variable & Distribution \\ 
		\hline
		$Q$: annual maximum flow rate & $Gumbel(mode=1013, scale=558)$ \\
		$K_s$: Manning-Strickler coefficient & $\mathcal{N}(\overline{x}=30,\sigma=7.5)$ \\
		$Z_v$: Depth measure of the river downstream & $\mathcal{U}(49,51)$ \\
		$Z_m$: Depth measure of the river upstream & $\mathcal{U}(54,55)$  \\
		\hline
	\end{tabular}
	\caption{Initial distribution of the 4 inputs of the hydraulic model.}
	\label{tab:initial distribution hydraulic model}
\end{table}
\begin{table}[ht]
	\centering
	\begin{tabular}{|l|c|c|c|c|}
		\hline
		Variable & Bounds & Mean & \specialcell{Second order \\ moment} & \specialcell{Third order \\ moment} \\ 
		\hline
		$Q$ & $[160, 3580]$ & $1320.42$ & $2.1632 \times 10^6$ & $4.18\times 10^9$ \\
		$K_s$ & $[12.55, 47.45]$ & $30$ & $949$ & $31422$ \\
		$Z_v$ & $[49,51]$ & $50$ & $2500$ & $125050$ \\
		$Z_m$ & $[54,55]$ & $54.5$ & $2970$ & $161892$ \\
		\hline
	\end{tabular}
	\caption{Corresponding moment constraints of the 4 inputs of the hydraulic model.}
	\label{tab : Constraints for hydraulic model}
\end{table}
\noindent The height of the river $H$ is calculated through the analytical model
\begin{equation}
	H = \left(\frac{Q}{300K_s \sqrt{\frac{Z_m - Z_v}{5000}}} \right)^{3/5} \ .
	\label{eq:Hydraulic Model}
\end{equation}
We are interested in the flood probability $\sup_{\mu\in\mathcal{A}} P(H \geq h)$.

\subsection{Maximum constraints order influence}

We will compare the influence of the constraint order on the optimum. The initial distributions and the constraints enforced are available in Table \ref{tab : Constraints for hydraulic model}. The value of the constraints correspond to the moments of the initial distributions. Figure \ref{fig : Moment Order Constraints} shows how the size of the optimization space $\mathcal{A}$ decreases by adding new constraints. A differential evolution solver was used to perform the optimization. The initial CDF was computed with a Monte Carlo algorithm. One can observe that enforcing only one constraint on the mean will give a robust quantile significantly larger than the one of the initial distribution. On the other hand, adding three constraints on every inputs reduces quite drastically the space so that the robust quantiles found are closed to the one of the initial CDF.

\begin{figure}[H]
	\centering
	\clipbox{0pt 0pt 0pt 18.2pt}{\includegraphics[scale=0.35]{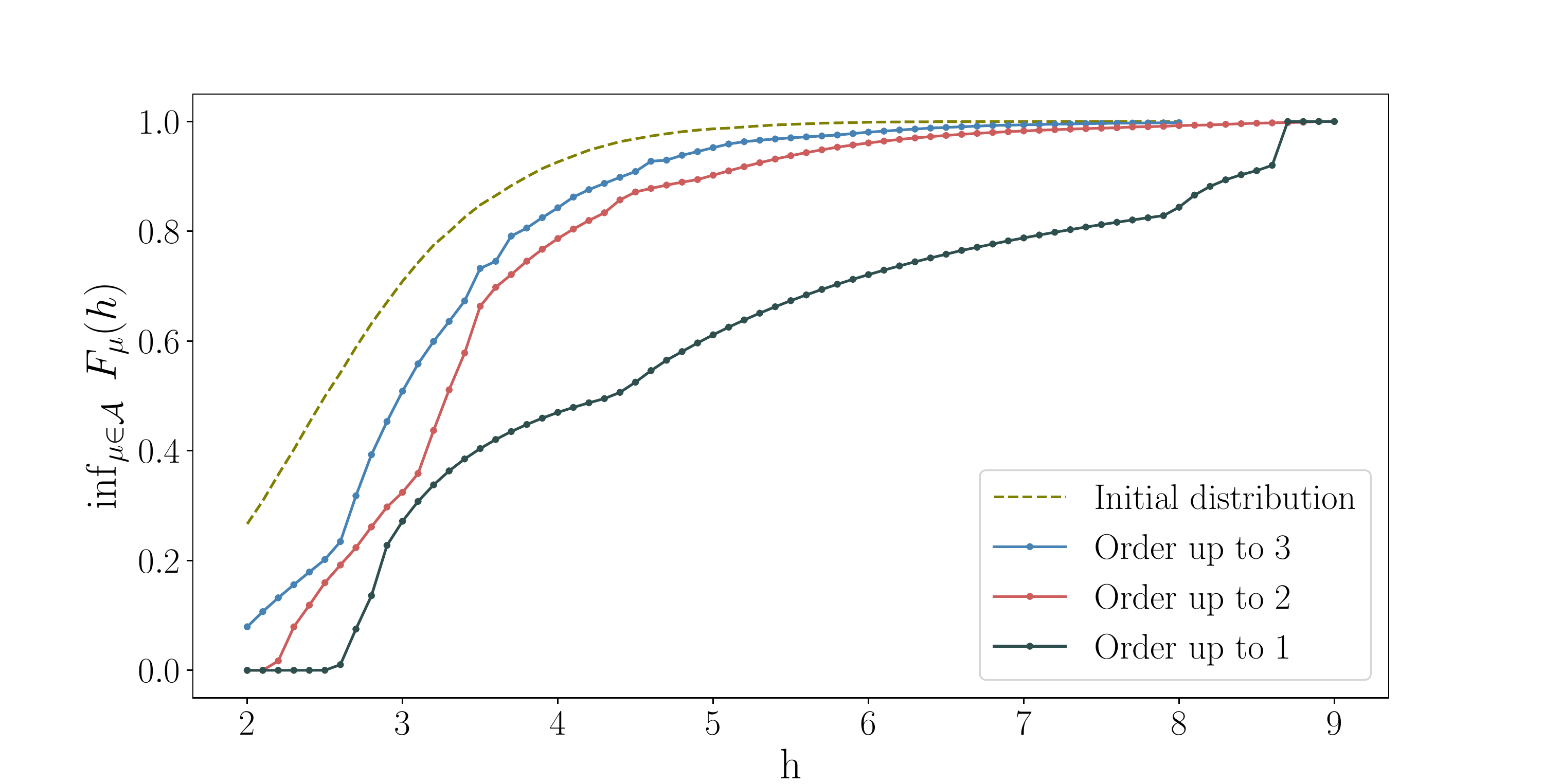}}
	\caption{Influence of the number of constraints enforced on the minimal CDF.} 
	\label{fig : Moment Order Constraints} 
\end{figure}

\subsection{Comparison with the Mystic framework}
We highlight the interest of the canonical moments parameterization by comparing its performances with the Mystic framework \citep{mckerns_optimal_2012}. Mystic is a Python toolbox suitable for OUQ. In Figure \ref{fig : Comparison Canonical Mystic HydraulicModel} one can see the comparison beween Mystic and our algorithm. Both computation were realized with an identical solver, and computational times were similar ($\approx$30 min). We enforced one constraint on the mean of each input (see Table \ref{tab : Constraints for hydraulic model}). The performance of the Mystic framework is outperformed by our algorithm. Indeed, the generation of the weights and support points of the input distributions is not optimized in the Mystic framework. Hence, an intermediary transformation of the measure is needed in order to respect the constraints. During this transformation, the support points can be send out of bounds so that the measure is no more admissible. Many population vectors are rejected, which reduces the overall performance of the algorithm. Meanwhile, our algorithm warrants the exploration of the whole admissible set of measure without any vector rejection. 

\begin{figure}[htb]
	\centering
	\clipbox{0pt 0pt 0pt 18.2pt}{\includegraphics[scale=0.35]{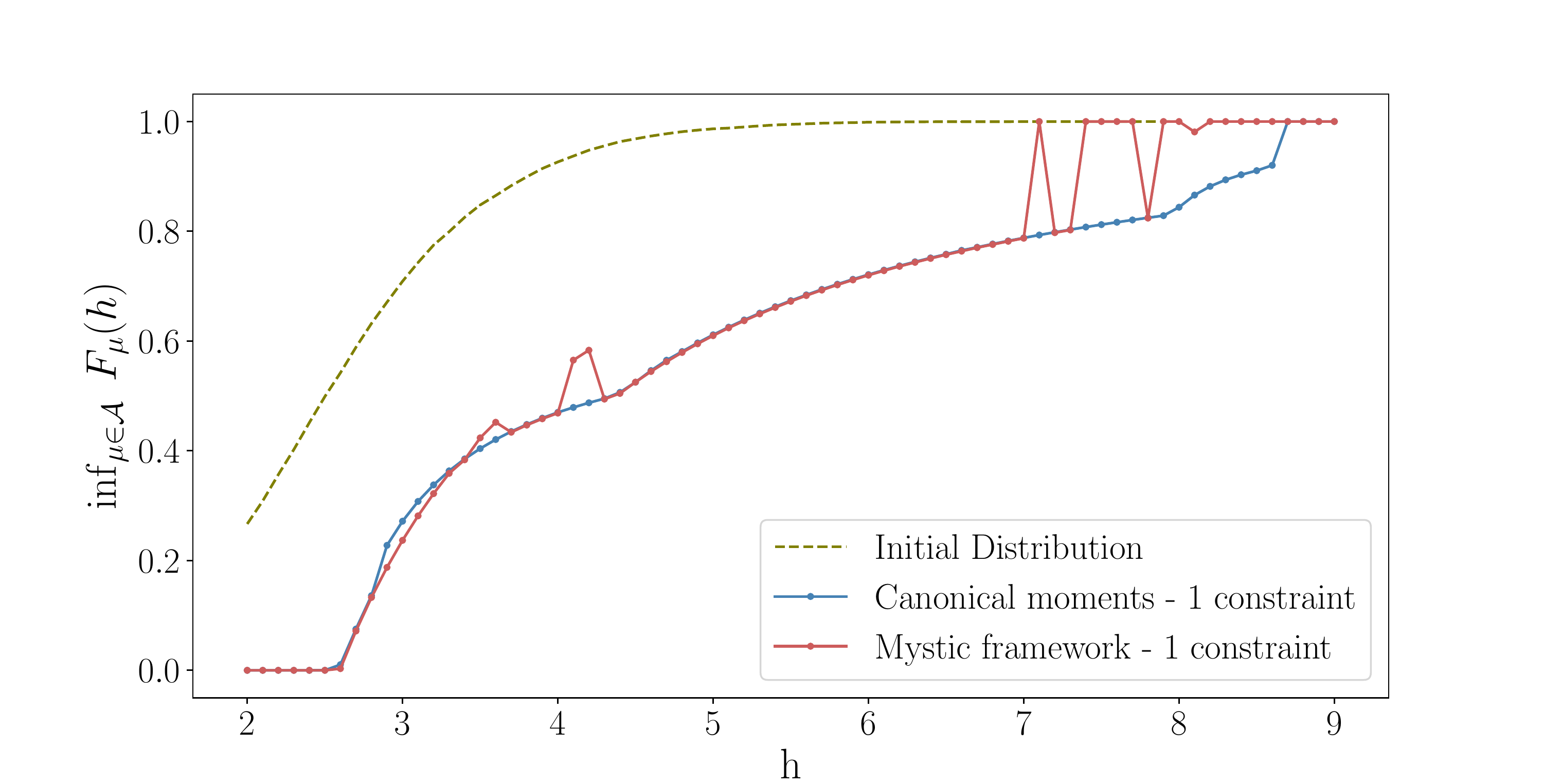}}
	\caption{Comparison of the performance of the Mystic framework and our algorithm on the hydraulic code.} 
	\label{fig : Comparison Canonical Mystic HydraulicModel} 
\end{figure}

\section{Real case study}
\label{sec: Real case study}
The CATHARE code simulates the pressure and temperature inside a simplified reactor during a cold leg Intermediate Break Loss Of Coolant Accident (IBLOCA) \citep{iooss_advanced_2018}. An output of the code CATHARE is time dependent and is depicted in Figure \ref{fig:RAW CATHARE}. In this work, the variable of interest is only the maximal temperature inside the reactor. Ultimately we wish to obtain a robust quantile of high order (around 0.95-0.99) of this quantity.
Two test cases have been realized on a water pressured reactor built at smaller scale (1:1 in height and 1:48 in volume, Figure \ref{fig:REP maquette}). The core of the replica is heated with an electric heater, while the breach happens to be on the cold leg section. The first breach size was 17$\%$ of the section, and the second 13$\%$. 

The CATHARE code takes 27 inputs in its simplified version, all representing physical parameters and assumed mutually independent. One run of the code takes approximately 20 minutes which makes it very costly for our purpose. Because we use the code as black-box, the use of a surrogate model is mandatory.

\begin{figure}[htb]
	\begin{minipage}[t]{0.48\linewidth}
		\centering
		\includegraphics[scale=0.18]{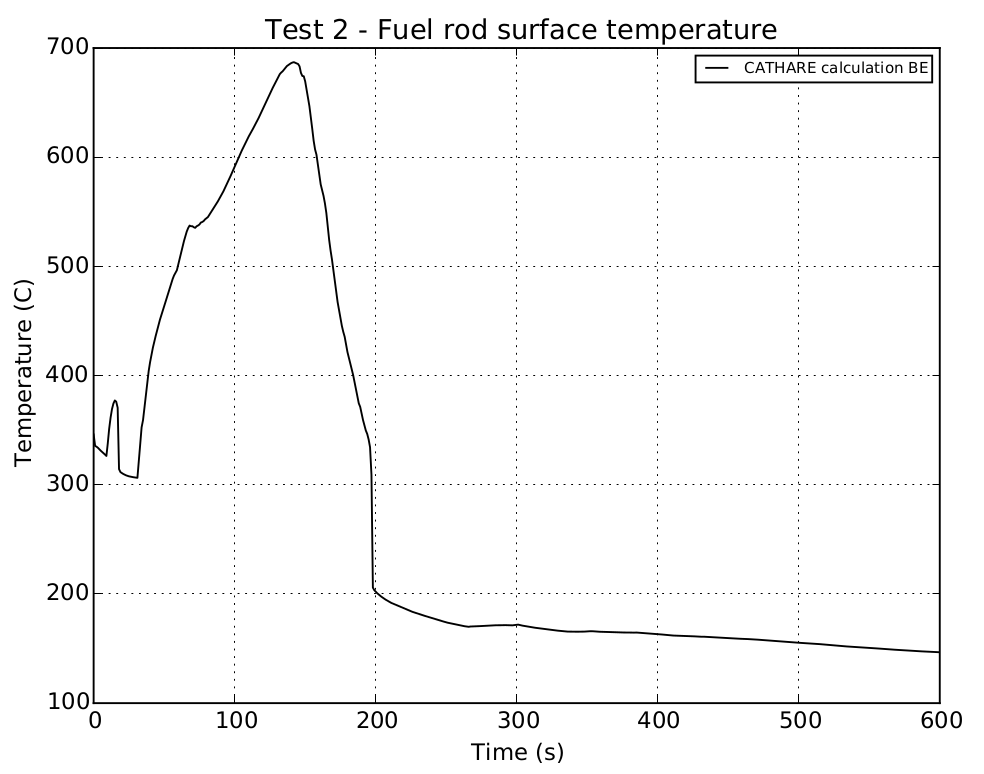}
		\caption{CATHARE temperature output for nominal parameters, the maximal temperature value is 687$^\circ$C after 140 secondes.}
		\label{fig:RAW CATHARE}
	\end{minipage}
	\hfill
	\begin{minipage}[t]{0.48\linewidth}
		\centering
		\includegraphics[scale=0.31]{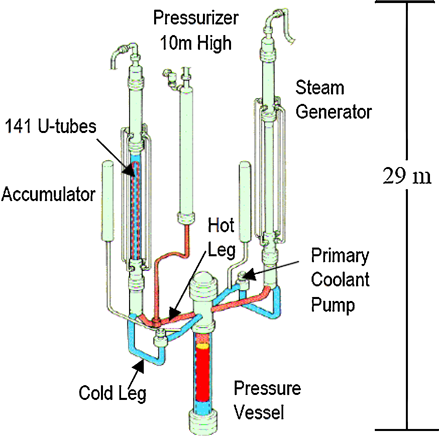}
		\caption{The replica of a water pressured reactor, with the hot and cold leg.}
		\label{fig:REP maquette}
	\end{minipage}
\end{figure}  

Gaussian process regression, also known as Kriging \citep{rasmussen_gaussian_2005}, makes the assumption that the response of the complex code is a realization of a Gaussian process, conditioned by some code observations. This approach provides the basis for statistical inference. We used the Openturns software \citep{baudin_title_2017} to compute the surrogate model. The covariance kernel of our Gaussian process was an anisotropic Mat\'ern $5/2$ one. We used an available Monte Carlo sample of 1000 code evaluations to condition our metamodel. The analytic predictivity coefficient \citep{gratiet_metamodel-based_2017} is equal to $Q^2=0.92$. 

\begin{table}[ht]
	\centering
	\begin{tabular}{|l|c|c|c|c|}
		\hline
		Variable & Bounds & \specialcell{Initial distribution \\ (truncated)} & Mean & \specialcell{Second order \\ moment} \\ 
		\hline
		$n^\circ 10$ & $[0.1, 10]$ & $LogNormal(0,0.76)$ & $1.33$ & $3.02$ \\
		$n^\circ 22$ & $[0, 12.8]$ & $Normal(6.4,4.27)$ & $6.4$ & $45.39$ \\
		$n^\circ 25$ & $[11.1, 16.57]$ & $Normal(13.79$ & $13.83$ & $192.22$ \\
		$n^\circ 2$ & $[-44.9, 63.5]$ & $Uniform(-44.9,63.5)$ & $9.3$ & $1065$ \\
		$n^\circ 12$ & $[0.1, 10]$ & $LogNormal(0,0.76)$ & $1.33$ & $3.02$ \\
		$n^\circ 9$ & $[0.1, 10]$ & $LogNormal(0,0.76)$ & $1.33$ & $3.02$ \\
		$n^\circ 14$ & $[0.235, 3.45]$ & $LogNormal(-0.1,0.45)$ & $0.99$ & $1.19$ \\
		$n^\circ 15$ & $[0.1, 3]$ & $LogNormal(-0.6,0.57)$ & $0.64$ & $0.55$ \\
		$n^\circ 13$ & $[0.1, 10]$ & $LogNormal(0,0.76)$ & $1.33$ & $3.02$ \\
		\hline
	\end{tabular}
	\caption{Corresponding moment constraints of the 9 most influential inputs of the CATHARE model.}
	\label{tab : Constraints for CATHARE model}
\end{table}

We already mention the fact that the main cost of the computation is due to the high number of code calls. The code needs to be evaluated on every point of the $p$-dimensional grid (see Equation \eqref{eq : Probability of faillure sum of weights}). The growth of the grid is exponential with the dimension. When $N_i$ constraint are enforced on every input $\mu_i$ for $1\leq i \leq p$, the size of the grid is exactly $\prod_{i=1}^p (N_i + 1)$. So that one estimation of the \textsc{p.o.f} can already be time consuming. It is realistic to say that the overall methodology must be limited to dimensions lower than 10. 
Because the CATHARE code takes 27 parameters, we therefore applied a screening strategy (reproduce in \cite{iooss_advanced_2018}) and highlighted 9 of the most influential parameters.

\begin{figure}[htb]
	\centering
	\clipbox{0pt 0pt 0pt 18.2pt}{\includegraphics[scale=0.35]{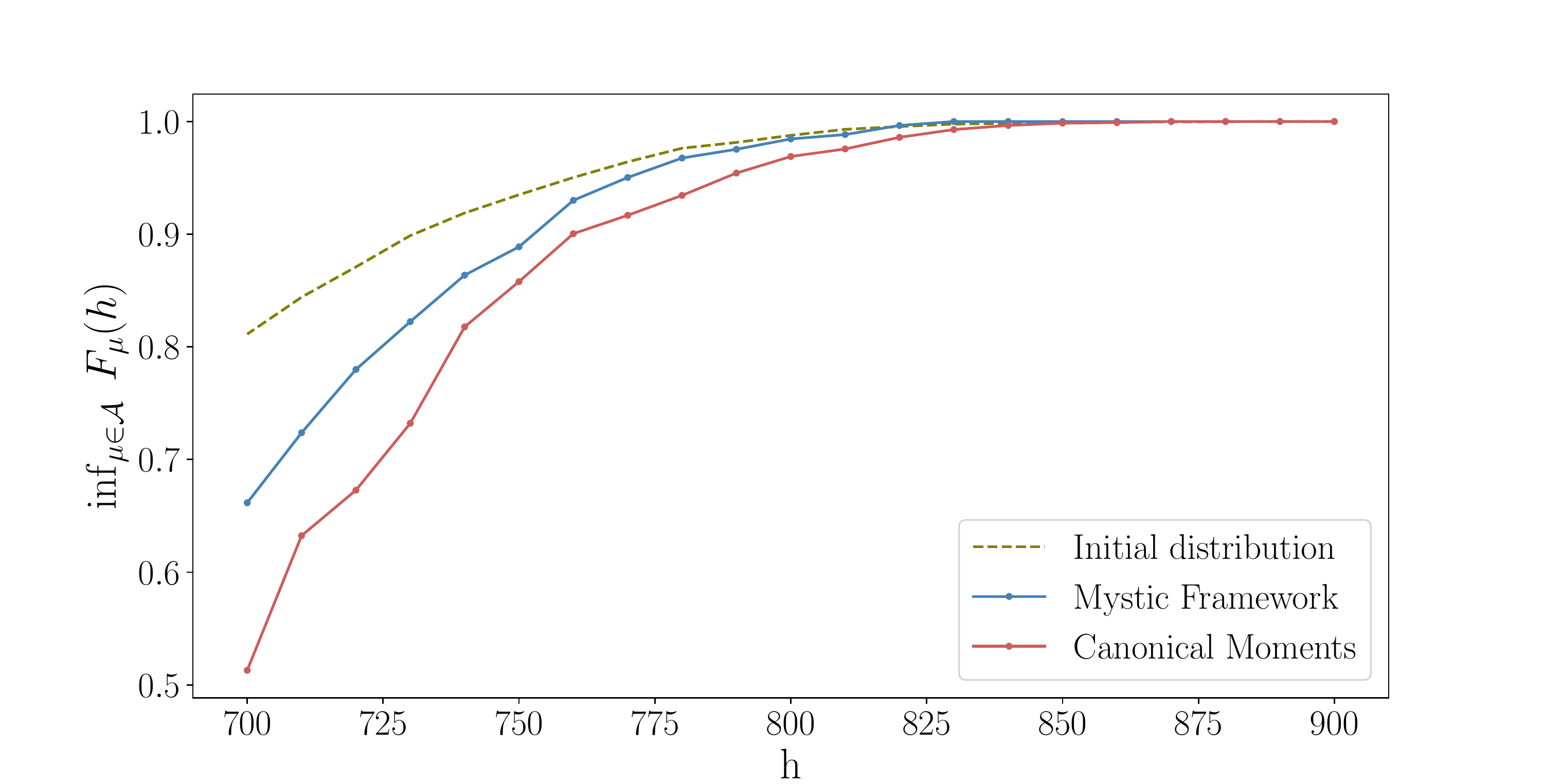}}
	\caption{Comparison of the performance and our algorithm on the 9 dimensional restricted CATHARE code.} 
	\label{fig : Comparison Canonical Mystic CATHARE Model} 
\end{figure}

Two constraints were enforced on their first two moments as displayed in Table \ref{tab : Constraints for CATHARE model}. We successfully applied the methodology on the 9 dimensional restricted code. However, the computation was one day long for each threshold. We restricted the computation of the CDF to a small specific area of interest (high quantile 0.5-0.99) and we parallelized the task so that the computation did not exceed one week. One can compare, in Figure \ref{fig : Comparison Canonical Mystic CATHARE Model}, the results of the computation realized with the Mystic framework and our algorithm. It confirms the difficulty for the Mystic framework to explore the whole space of admissible measures. On the other hand, this proves the efficiency of canonical moments to solve this optimization problem.

\section{Conclusion}
\label{sec: Conclusion}

We successfully adapted the theory of canonical moment into an improved methodology for solving OUQ problems. The restriction to moment constraints suits most of the practical engineering cases. Our algorithm shows very good performances and great adaptability to any constraints order. The optimization is subject to the curse of dimension and should be kept under 10 parameters, else dimension reduction strategies should be considered. The use of a metamodel in our real case study introduces a new source of uncertainty. In further work we will study how to take it into account and quantify its impact on the quantile estimation (see for instance \cite{bect_sequential_2012}). In practical engineering cases, the inputs correspond to physical parameters, however, the joint distribution of the optimum is a discrete measure. One can criticize that it hardly corresponds to a real word situation. In order to address this issue, we will search for new optimization sets whose extreme points are not discrete measure. One can find in the literature of robust bayesian analysis the $\varepsilon$-contamination class or the unimodal class, that might be of some interest in this situation. New measures of risk will also be explored, such as superquantile \citep{rockafellar_random_2014}.

\appendix
\section*{Appendix}
\subsection{Proof of duality proposition \ref{THM : DUALITY THEOREM}}\label{app : proof}

\begin{proof}
	we denote by $\displaystyle a = \sup_{\mu \in \mathcal{A}} \bigg[ \inf \left\lbrace h \in \mathbb{R}\ ; \ F_{\mu}(h) \geq p \right\rbrace\bigg]\ $ and $\displaystyle\ b = \inf \left\lbrace h  \in \mathbb{R}\; | \; \inf_{\mu\in\mathcal{A}} F_{\mu}(h) \geq p \right\rbrace$. In order to prove $a=b$, we proceed in two step. First step, we have
	\begin{eqnarray*}
		 & & \text{for all } h \geq b \ ;\  \inf_{\mu\in \mathcal{A}} F(\mu;h)  \geq p\ ,\\
		\Leftrightarrow & & \text{for all } h \geq b \ \text{ and for all } \mu\in\mathcal{A}\ ;\  F(\mu;h)  \geq p\ ,\\
		\Leftrightarrow & & \text{for all } \mu \in \mathcal{A}\ \text{ and for all } h \geq b \ ;\   F(\mu;h)  \geq p\ ,\\
		\Rightarrow & & \text{for all } \mu \in \mathcal{A}\ ;\ \inf\{h \in \mathbb{R} \, | \,  F(\mu;h) \geq p\}  \leq b \ ,
	\end{eqnarray*}
	so that $b \geq a$. Second step, because $a$ is the sup of the quantile,
	\begin{eqnarray*}
			& & \text{for all } h\geq a\ ;\ \text{ for all }\mu\in\mathcal{A}\ ;\   F(\mu; h)  \geq p\ ,\\
			\Rightarrow & &  \text{for all } h\geq a\ ;\   \inf_{\mu\in\mathcal{A}} F(\mu; h)  \geq p\ ,
	\end{eqnarray*}
	so that
	\[\inf\ \bigg[ h  \in \mathbb{R} \ | \ \inf_{\mu \in \mathcal{A}}  F_{\mu}(h) \geq p \bigg]   \leq a \ ,\]
	and $b\leq a$.
\end{proof}	

\bibliography{Biblio.bib}

\end{document}